%% file: FormalStability_arxiv.tex
\documentclass[12pt,a4paper]{amsart}

\usepackage{standard}
\usepackage{natbib}
\usepackage{microtype}
\newcommand{\Verification}{Verification}

\begin{document}
\title{Foundations of Machine-Checked Control Theory in Lean}

\author[M. Doll]{Moritz Doll}
\address{Department of Electrical and Electronic Engineering \\ University of Melbourne \\ Parkville, VIC-3010, Australia}
\email{moritz.doll@unimelb.edu.au}

\author[I. Shames]{Iman Shames}
\address{Department of Electrical and Electronic Engineering \\ University of Melbourne \\ Parkville, VIC-3010, Australia}
\email{iman.shames@unimelb.edu.au}

\input{abstract}

\maketitle

\input{FormalStability_main}

\bibliographystyle{abbrv}
\bibliography{leanbib}

\end{document}

%% file: abstract.tex
\begin{abstract}
We introduce an open-source library for machine-checked control theory in the
interactive proof assistant Lean to lay foundations for the verification of
cyber-physical systems.
To this end, as representative theorems, we present
formalizations of Lyapunov stability theory and the small-gain theorem.
First,
the machinery employed for formalizing Lyapunov stability, i.e., neighborhood
filters, allows stating a Lyapunov theorem that covers both points and sets and applies to
continuous, discrete, and hybrid systems.
Second, the small-gain theorem is
proved via stating input-output systems as relations without the
usual well-posedness assumption.
The Lean formalization of each of these
theorems is then presented.
We conclude by discussing the library architecture
and mentioning some of the other system theoretic results that are formalized in
the library along with future plans.
\end{abstract}

%% file: FormalStability_main.tex
\section{Introduction}
\Verification{} of cyber-physical systems using proof assistants
has gained considerable interest in recent years.
However, the use of proof
assistants has been limited by the available formalized mathematics and most
foundational implementations have only been able to formalize certain aspects,
such as LaSalle's invariance principle~\cite{CoRo} or the small-gain
theorem~\cite{JaVe}.
The recent interest of mathematicians in interactive proof
assistants, in particular Lean~\cite{Lean} and its mathematical library
\mathlib~\cite{Mathlib}, provides an opportunity to build a comprehensive
library of formalized control theory.
Moreover, the success of \mathlib has
sparked the development of a similarly designed library for computer science,
CSLib~\cite{CSLib}.
While CSLib is only in its early stage of development, it
aims to provide tools to verify real-world computer programs.

We introduce a library of formalized dynamical systems and control theory that
aims to build the foundations for verification of cyber-physical systems.

To keep the presentation self-contained, we recall
the ideas behind Lean, the role of tactics, and what constitutes a formal
proof.
Lean is built on dependent type theory
%(a variant of the calculus of inductive constructions)
and rests on the Curry--Howard correspondence: every
proposition is represented as a type, and a proof of that proposition is a term
inhabiting that type \cite{Wadler}.
Verifying a proof therefore reduces to type-checking a
term, a task delegated to a small, trusted kernel, so correctness depends only
on this minimal core and not on the much larger elaboration and automation
machinery layered above it.
Lean is at the same time a general-purpose
functional programming language, so its own automation is written in Lean
itself, and \mathlib supplies not only a large base of formalized definitions and
results, but also domain-specific automation.

Constructing such proof terms by hand is tedious, so proofs are usually built
with tactics, commands that manipulate the proof state.
At any point this state
consists of one or more goals, each a target type to be inhabited together with
a local context of hypotheses, and a tactic either closes a goal or replaces it
with simpler subgoals.
For instance, \lstinline{intro} moves a hypothesis from
the goal into the context, \lstinline{apply} reasons backward through a lemma,
\lstinline{exact} supplies a term outright, and \lstinline{rw} and
\lstinline{simp} rewrite and simplify, while more sophisticated tactics like
\lstinline{ring} and \lstinline{linarith} prove equalities in rings and linear inequalities,
respectively.
Tactics are themselves Lean programs, so users can define new ones,
yet every tactic block still elaborates to a proof term that the kernel checks,
so this automation adds convenience without compromising correctness.

A formal proof in Lean is therefore, in the end, a term whose type is the
statement being proved.
A theorem declaration consists of a name, a statement
expressed as a type with its quantified variables and hypotheses appearing as
parameters, and a body that produces a term of that type.
The body is given
either directly as a term or, more commonly, through a sequence of tactics,
drawing on previously established definitions and lemmas.
Once elaborated, the
complete term is passed to the kernel for type-checking, and its acceptance is
exactly the guarantee that the stated result holds.

While Lean checks the correctness of \emph{theorems}, it should be noted
that it does not provide any guarantees about definitions.
In particular, by a slight
change to a definition a wrong statement might become provable in Lean.
It is therefore
of utmost importance that the definitions are formalized faithfully.

Safety arguments for cyber-physical systems rest on stability certificates: a
Lyapunov function that keeps the state bounded, a forward-invariant safe set the
state cannot leave, or a finite gain bound that limits how much a subsystem
amplifies disturbances.
For safety, a Lean formalization represents these
certificates as precise, reusable objects, each paired with a theorem that
states exactly when it implies the stability or invariance property, including
the set and interconnection cases that textbook definitions often omit.
For
trust, it changes what must be believed.
The certificate is re-checked by Lean's
small kernel against the formal definitions.
Its validity therefore no longer
depends on a numerical solver's floating-point arithmetic or a hand
proof.
The check is reproducible and is re-run automatically whenever
the system changes.

In this paper, we introduce the library and, taking Lyapunov stability and the
small-gain theorem as two representative examples, outline how such results are
formalized in Lean and which mathematical concepts the formalization relies on.
To this end, for each result, we first introduce the mathematical machinery
needed to understand its formalization.
Then drawing from this insight, we state
potentially new proofs and show how they are formalized.

While this library is to our knowledge the first formalization of stability
theory in Lean, it is far from the first formalization of stability theory in
any theorem prover.
In fact, Cohen--Rouhling~\cite{CoRo} formalized LaSalle's
invariance principle in Rocq (formerly Coq), and Jasim--Veres~\cite{JaVe} formalized the
small-gain theorem in Isabelle/HOL.
We also mention VeriDrone~\cite{VeriDrone} and KeYmaeraX~\cite{KeYmaeraX} that
formalizes dynamical systems and control theory.
However, KeYmaeraX uses
so-called differential dynamic logic and external tools such as Mathematica
and SMT solvers.
VeriDrone is built in Rocq, but also outsources some of the
proofs to SMT solvers such as Z3.

The Lean approach of this paper enjoys three advantages.
First, \mathlib is a single integrated foundation, so Lyapunov
theory and input-output theory share the same definitions and compose without a
translation layer, whereas Rocq's analysis is split across competing libraries 
and Isabelle's is spread across HOL-Analysis and the federated AFP.
Second, \mathlib's
abstractions (filters, relations, arbitrary index sets) let one statement cover
points and sets, continuous, discrete, and hybrid systems, and even drop
well-posedness in the small-gain theorem, so a general library replaces the
one-theorem developments.
Third, everything is checked by Lean's small kernel,
so the trusted base is the kernel and the definitions, unlike the tools that
must trust, or reconstruct, the output of Z3 or Mathematica.

Beyond the mathematics, Lean is also a
programming language.
Thus, control-specific automation and kernel-checked,
CI-verified documentation (Verso) live in the same system.
The contributor base
is large and active.
Gaps the library fills can be pushed upstream to benefit
everyone, and the sibling CSLib effort points toward verified embedded
controllers within one ecosystem.
This paper along with the proposed library lays the foundation for a general,
trusted, and extensible control-theory library that prior efforts, by
construction, do not provide.

\section{Fundamental stability results in Lean}
The library \cite{LeanDynamicalSystems} is hosted on GitHub.
The two main parts of the library are currently Lyapunov stability theory and input-output
stability.
We will give an outline of the topics covered and the specific design decisions that
differentiate the formalization from a usual textbook presentation.

\subsection{Stability of dynamical systems}
Lyapunov stability is a classical topic in control theory.
We start by recalling the typical presentation found in textbooks such as
Khalil~\cite{Khalil} and then describe the path taken for formalization in Lean.

We start by considering the initial value problem of an autonomous system,
\begin{equation}
        \dot{x}(t) = f(x)\,,\quad 
        x(0) = \bar{x}_0\,,
\end{equation}
where $f : \RR^n \to \RR^n$ is locally Lipschitz continuous and $\bar{x}_0 \in \RR^n$.
Then by the Picard--Lindelöf theorem there exists a unique solution $x :
(-\eps_0, \eps_1) \to \RR^n$.
We will assume that solutions exist for all time and denote the solution map by
$\Phi : \RR \times \RR^n \to \RR^n$.
If $f(x_0) = 0$, then the solution is given by $\Phi(t, x_0) = x_0$.
\begin{definition}\label{def:lyapunov}
    The fixed point $x_0$ is called \emph{Lyapunov stable} if for every $\eps >
    0$ there exists a $\delta > 0$ such that for all $x$ with $d(x, x_0) < \delta$,
    $d(\Phi(t, x), x_0) < \eps$ for all $t \geq 0$, where $d(x, y) = \norm{x - y}$ denotes the
    Euclidean distance.

    The fixed point is called \emph{asymptotically stable} if it is Lyapunov stable and there
    exists a neighborhood of $x_0$ such that all trajectories starting in that neighborhood
    converge to $x_0$.
\end{definition}
Lyapunov's direct method gives a criterion to check stability: if there exists a
function $V : \RR^n \to \RR$ that is continuous, $V(x_0) = 0$ and $V(x)>0$ for $x
\ne x_0$, and $V$ is non-increasing (decreasing) along the
flow, then $x_0$ is stable (asymptotic stable).
If $V$ is differentiable, one
can check the decrease along the flow by calculating the Lie derivative: $\pa_t
V(\Phi(t,x)) = \ang{\pa_x V(\Phi(t,x)), f(\Phi(t,x))}$.
In particular, if
$\ang{\pa_x V(x), f(x)} < 0$ for all $x \not = x_0$, then it follows that $x_0$
is asymptotic stable.
In the case that $\ang{\pa_x V(x), f(x)} \leq 0$, under
some circumstances it is still possible to deduce asymptotic stability by
invoking LaSalle's invariance principle.

A good formalization has to be sufficiently general so that it can be applied in various situations
and we immediately observe that the above definition of stability is unnecessarily restrictive.
The above definition does not apply to discrete dynamical systems, such as iterated function
systems, and does not cover stability of sets.
In order to generalize the definition to cover both
continuous and discrete systems, we note that the only important object coming from the
differential equation is the solution operator $\Phi$.
In particular, we can take an arbitrary $\Phi : \iota \times X \to X$ as long as we assume that
$\iota$ is a preordered set and $X$ is a topological space.
The generalization to cover both stability of points and sets is more involved.
The first step to observe is that Lyapunov stability is qualitative in the sense that it does not
specify how large $\delta$ is in comparison to $\eps$.
Hence, using standard point set topology, one can rephrase stability as for every neighborhood $U$
of $x_0$ there exists a neighborhood $W$ of $x_0$ such that for all $t$ and $x \in W$, it follows
that $\Phi(t,x) \in U$.
This definition can be generalized even further to use \emph{filters}.
\begin{definition}
    Let $X$ be any set.
    A \emph{filter} $\mathfrak{F}$ on $X$ is a set of subsets of $X$ such that
    \begin{itemize}
        \item every subset of $X$ that contains an element in $\mathfrak{F}$ belongs to
            $\mathfrak{F}$,
        \item the finite intersection of elements in $\mathfrak{F}$ belongs to $\mathfrak{F}$, and
        \item $X \in \mathfrak{F}$.
    \end{itemize}
\end{definition}
Note that Bourbaki~\cite{Bourbaki} does not assume the third condition, but assumes that the empty
set is not contained in $\mathfrak{F}$, which implies that $\mathfrak{F} \ne \mathcal{P}(X)$, where
$\mathcal{P}(X)$ is the power set of $X$.
\begin{definition}
    Let $\mathfrak{F}$ be a filter on $X$.
    A filter basis of $\mathfrak{F}$ is a family of subsets $\mathfrak{B}\subset \mathcal{P}(X)$
    such that for every $A \in \mathcal{P}(X)$ we have that
    $A \in \mathfrak{F}$ if and only if there exists a $B \in \mathfrak{B}$ such that $B \subset A$.
\end{definition}
If $X$ is a topological space and $x \in X$, then the \emph{neighborhood filter},
$\mathcal{N}(x)$ is given by all sets $U$ such that there exists an open subset $U_0$ with
$x \in U_0 \subset U$, that is, $\mathcal{N}(x)$ is the filter generated by open subsets containing
$x$.
This leads to the following notion of stability of a filter along a map:
\begin{definition}
    Let $\iota$ and $X$ be arbitrary sets, $I \subset \iota$, and $\Phi : \iota \times X \to X$.
    A filter $\mathfrak{F}$ on $X$ is called \emph{stable} along $\Phi$ on $I$ if for all
    $S \in \mathfrak{F}$ there exists an $S' \in \mathfrak{F}$ such that for all $t \in I$ and all
    $x \in S'$ it follows that $\Phi(t,x) \in S$.
\end{definition}
We note that another way of generalizing Lyapunov stability is by using set-valued maps, see
\cite{BiPa2023} and \cite{BiAn2025}.

In Lean this definition takes the form:

\begin{lstlisting}
variable {l : Filter X} {Φ : ι → X → X} {I : Set ι}

variable (Φ l I) in
def IsStableOn : Prop :=
  ∀ s ∈ l, ∃ s' ∈ l, ∀ t ∈ I, ∀ x ∈ s', Φ t x ∈ s
\end{lstlisting}

In the case that we take $l$ to be the neighborhood filter of $x$ and $I$ the
interval $\ico{0, \infty}$ we recover the typical textbook definition of
stability:

\begin{lemma}\label{lem:stable}
    Let $X$ be a metric space and $x_0 \in X$.
    The neighborhood filter $\mathcal{N}(x_0)$ is stable
    along $\Phi$ on $\ico{0,\infty}$ if and only if $x_0$ is Lyapunov stable.
\end{lemma}
\begin{proof}
    This follows directly from the fact that open balls of positive radius around $x_0$ form a
    basis of the neighborhood filter.
\end{proof}
In Lean this equivalence is formalised as the following where the interval $\ico{0, \infty}$ is
denoted by \lstinline{Ici 0}.
\begin{lstlisting}
example (x₀ : X) : (𝓝 x₀).IsStableOn Φ (Ici 0) ↔
    ∀ ε > 0, ∃ δ > 0, ∀ t ≥ 0, ∀ x,
    d(x, x₀) < δ → d(Φ t x, x₀) < ε
\end{lstlisting}
The main advantage of the definition of stability in terms of an arbitrary filter is that it
applies equally well in the case of stability to a compact subset $K \subset X$.
The relevant filter is the neighborhood filter of sets, $\mathcal{N}^s(K)$ defined as the
collection of all sets that contain an open set $s_0$ with $K \subset s_0$.
Similarly, we obtain a simple description of stability in this case:
\begin{lemma}\label{lem:stable_set}
    Let $X$ be a metric space, $K \subset X$ compact, and $K \ne \emptyset$.
    The neighborhood filter $\mathcal{N}^s(K)$ is stable along $\Phi$ on $\ico{0,\infty}$ if and
    only if for all $\eps > 0$ there exists $\delta > 0$ such that for all $t \geq 0$ and $x \in X$
    with $d(x, K) < \delta$ it follows that $d(\Phi(t,x), K) < \eps$, where
    $d(x, K) \coloneqq \inf_{y \in K} d(x, y)$ is the distance between a point and a set.
\end{lemma}
\begin{proof}
    The argument is the same as for Lemma~\ref{lem:stable} noting that the sets
    $K_\delta \coloneqq \set{x \in X \colon d(x, K) < \delta}$ form a basis for $\mathcal{N}^s(K)$
    if $K$ is compact.
\end{proof}
In Lean, this lemma takes the form:
\begin{lstlisting}[mathescape]
example (hK : IsCompact K) (hK' : K.Nonempty) :
    (𝓝${}^s$ K).IsStableOn Φ (Ici 0) ↔
    ∀ ε > 0, ∃ δ > 0, ∀ t ≥ 0, ∀ x,
    infDist x K < δ → infDist (Φ t x) K < ε
\end{lstlisting}
\begin{definition}
    A topological space $X$ is called \emph{first-countable} if for every $x$ the neighborhood
    filter $\mathcal{N}(x)$ admits a countable basis.
\end{definition}
In particular, every metric space is first-countable.

\begin{theorem}
    Let $X$ be a first-countable topological space and $I$ be a preordered set and fix $t_0 \in I$.
    Let $\Phi : I \times X \to X$ be any map such that $\Phi(t_0, x) = x$ for
    all $x \in X$ and $K \subset X$.
    If there exists a  function $V : X \to \RR$ such that
    \begin{enumerate}[label=(\emph{\roman*}), ref=(\emph{\roman*})]
        \item\label{ass_item:cont} $V$ is continuous,
        \item\label{ass_item:stationary} $V(x) \geq 0$ for all $x \in X$ and $V(x) = 0$ if and only
            if $x \in K$, 
        \item for each $x$ the map $t \mapsto V(\Phi(t, x))$ is non-increasing,
        \item\label{ass_item:sublevel} there exists $\delta_0 > 0$ such that the sublevel set
            $\set{x \in X \colon V(x) \leq \delta_0}$ is compact,
    \end{enumerate} 
    then the set $K$ is forward Lyapunov stable under $\Phi$, i.e., $\mathcal{N}^s(K)$ is
    stable along $\Phi$ on $\{t \in I : t \geq t_0\}$.
\end{theorem}
\begin{proof}
    Pick a $V$ that fulfills the assumptions.
    Let $\delta > 0$ and denote by
    $S_\delta \coloneqq \set{x \in X \colon V(x) \leq \delta}$ the sublevel sets
    of $V$.
    First we claim that $S_\delta$ is a basis for the filter
    $\mathcal{N}^s(K)$.
    For this we have to show that for every $\delta$,
    $S_\delta \in \mathcal{N}^s(K)$ and for every $S \in \mathcal{N}^s(K)$ there
    exists a $\delta > 0$ such that $S_\delta \subset S$. The first part follows
    directly from the fact that $V(x) = 0$ if and only if $x \in K$ and
    continuity of $V$.
    The second part is a proof by contradiction. Assuming for
    every $\delta > 0$, $S_\delta$ is not contained in $S$, we construct a
    sequence $a : \NN \to X$ such that $V(a(n)) \leq (n + 1)^{-1}$ and $a(n)
    \not \in S$ for all $n$.
    For all $n \geq N$, we have that $a(n) \in
    S_{\delta_0}$, so there exists (by compactness of $S_{\delta_0}$) a
    convergent subsequence $a(n_k)$ and $V(a(n_k)) \to 0$ as $k \to \infty$.
    This implies that $a(n_k) \to K$, which is a contradiction to the assumption
    that $a(n) \not \in S$ for all $n$ and $S$ being a neighborhood of $K$.

    The assumption that $V$ is non-increasing along $\Phi$ now implies that
    $\Phi$ maps $S_\eps$ to itself for $\eps \in (0,\delta_0)$, which proves the
    theorem.
\end{proof}
\begin{remark}
    The compactness of $K$ is a consequence of assumptions
    \ref{ass_item:cont}, \ref{ass_item:stationary}, and \ref{ass_item:sublevel}. 
\end{remark}
\begin{remark}
    This proof is theoretically interesting since most of the argument is purely about sublevel
    sets of $V$ and not about $\Phi$.
    In particular, at no stage was used that the map $\Phi$ is continuous,
    therefore this argument applies to hybrid systems, where the flow might be
    non-continuous.
\end{remark}
\begin{remark}
    This proof is easier to formalize than the usual proof, because it completely avoids any norm
    estimates and isolates the crucial part of the argument (the construction of the sequence $a$).
\end{remark}
The formalized version for sets reads in Lean as follows:
\begin{lstlisting}[mathescape]
theorem IsLyapunov.isStableOn_nhdsSet
    (h : IsLyapunov V Φ) (hVK : ∀ x, V x = 0 ↔ x ∈ K)
    (h_id : ∀ x, Φ t₀ x = x) {δ₀ : ℝ} (hδ₀ : 0 < δ₀)
    (h_cpt : IsCompact { x | V x ≤ δ₀ }) :
    (𝓝${}^s$ K).IsStableOn Φ (Ici t₀)
\end{lstlisting}
Here, \lstinline{IsLyapunov V Φ} is the proposition that \lstinline{V} is
continuous, non-negative, and is non-increasing along \lstinline{Φ}.

\subsection{Input-output maps}
The second part is analysis of input-output systems.
Throughout let $E, F$ be normed vector spaces.

Consider the space of locally $L^p$ functions, $L^p_{\loc}$.
An input-output system is simply a map
\begin{align}
    G : L^p_{\loc}(\RR_+; E) \to L^p_{\loc}(\RR_+; F)\,.
\end{align}
We also consider the truncation operator
\begin{align}
P_t : L^p_{\loc}(\RR_+; E) \to L^p(\RR_+; E)
\end{align}
given by
\begin{align}
    (P_t u)(\tau) = \begin{cases} u(\tau) & \tau \leq t \\ 0 & \text{otherwise} \end{cases}\,.
\end{align}

\begin{definition}
An input-output map is called \emph{finite gain stable} if there exist $\gamma, \beta \geq 0$ such
that for all $u \in L^p_{\loc}(\RR_+; E)$ and $t > 0$, we have
\begin{align}
    \norm{P_t G(u)}_{L^p} \leq \gamma \norm{P_t u}_{L^p} + \beta\,.
\end{align}
\end{definition}

The main advantage of considering input-output systems is to be able to prove stability of a
complicated system by decomposing it into smaller simple systems. It is easy to see that the
composition of two finite gain stable systems is again finite gain stable. The more complicated
connection is the negative feedback connection (see Fig.~\ref{fig:feedback}) and the result that
negative feedback connections are finite gain stable is the \emph{small-gain theorem}. First, we
will prove the small-gain theorem without assuming that the feedback connection is well-posed. This
means in particular, that we have to relax the definition of finite gain stability.
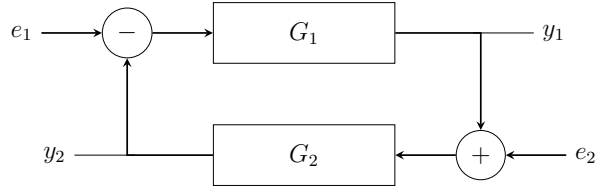
\begin{figure}
\centering
\scalebox{0.8}{\input{feedback2}}
\caption{A feedback connection of input-output maps $G_1$ and $G_2$.}
\label{fig:feedback}
\end{figure}

Recall that every function $f : X \to Y$ defines a subset of $X \times Y$, the \emph{graph}, given by
\begin{align*}
    \operatorname{graph}(f) = \set{(x, y) \colon y = f(x)}\,.
\end{align*}
We can then generalize the notion of finite gain stability to subsets of
$L^p_{\loc}(\RR_+; E) \times L^p_{\loc}(\RR_+; F)$ (which we will call relations in this context):
\begin{definition}
    A relation $R \subset L^p_{\loc}(\RR_+; E) \times L^p_{\loc}(\RR_+;F)$ is
    \emph{finite gain stable} if there exist $\gamma, \beta \geq 0$
    such that for all $(u, y) \in R$ and $t \geq 0$, we have
    \begin{align*}
        \norm{P_t y}_{L^p} \leq \gamma \norm{P_t u}_{L^p} + \beta\,.
    \end{align*}
\end{definition}
It is clear that the graph of a map is finite gain stable if and only if the map itself is finite
gain stable.
Given two maps $G_1, G_2$, we define the relations
\begin{align}
    R_1 &\coloneqq \set{ (e, u) \colon G_1(u_1) = u_2 - e_2,\, G_2(u_2) = e_1 -
    u_1}\,, \label{eq:R1}\\
    R_2 &\coloneqq \set{ (e, y) \colon G_1(e_1 - y_2) = y_1,\, G_2(e_2 + y_1) =
    y_2}\,, \label{eq:R2}
\end{align}
where $e = (e_1, e_2), u = (u_1, u_2)$, and $y = (y_1, y_2)$.
Well-posedness of the feedback connection is  the statement that $R_1$ and $R_2$ are the graphs of
maps.

Now, we state the small-gain theorem without the commonly stated well-posedness
assumption:
\begin{theorem}[Small-gain Theorem]\label{thm:small_gain}
    Let $G_1 : L^p_{\loc}(\RR_+; E) \to L^p_{\loc}(\RR_+; F)$ and $G_2 : L^p_{\loc}(\RR_+; F) \to L^p_{\loc}(\RR_+; E)$.
    Assume that for $j \in \set{1,2}$, $G_j$ is finite gain stable 
    with gain $\gamma_j$ and bias $\beta_j$.
    If $\gamma_1 \gamma_2 < 1$, then the relations $R_1$ and $R_2$ are finite gain stable.
\end{theorem}
\begin{proof}
    It is clear that if the relation $R_1$ is finite
    gain stable, then the relation $R_2$ is finite gain
    stable.
    Moreover, it suffices to show that there exist $\gamma, \beta$, such that for all
    $(e_1, e_2, u_1, u_2) \in R_1$, we have the estimate
    \begin{align*}
        \norm{P_t u_1}_{L^p} + \norm{P_t u_2}_{L^p} \leq \gamma (\norm{P_t e_1}_{L^p} + \norm{P_t e_2}_{L^p}) + \beta\,.
    \end{align*}

    We will now use the short-hand $u_t \coloneqq P_t u$ and omit the subscript of the norm.
    Using the equation of the feedback connection, $G_j$ being finite gain stable implies that for
    $u_1$ and $u_2$, we have
    \begin{align*}
        \norm{(u_1)_t} &\leq \norm{(e_1)_t} + \gamma_2 \norm{(u_2)_t} + \beta_2\,,\\
        \norm{(u_2)_t} &\leq \norm{(e_2)_t} + \gamma_1 \norm{(u_1)_t} + \beta_1\,.
    \end{align*}
    Inserting these equations into each other, we obtain
    \begin{align*}
        (1 - \gamma_1 \gamma_2) \norm{(u_1)_t} &\leq \norm{(e_1)_t} + \gamma_2 (\norm{(e_2)_t} + \beta_1) + \beta_2\,,\\
        (1 - \gamma_1 \gamma_2) \norm{(u_2)_t} &\leq \norm{(e_2)_t} + \gamma_1 (\norm{(e_1)_t} + \beta_2) + \beta_1\,.
    \end{align*}
    Since $1 - \gamma_1 \gamma_2 > 0$, the claimed inequalities follow.
\end{proof}
Note that this observation is not new (see for instance Zames~\cite{Zames} and van der
Schaft~\cite{vanDerSchaft}), but the formalization gets significantly simplified by using relations.

Now, we turn to the formalization of this version of the small-gain theorem in Lean.

First, we note that there are several different ways to discuss $L^p$ functions in \mathlib. Recall
that the space $L^p$ is defined not as functions, but equivalence classes of functions that are
equal almost everywhere (otherwise the norm would fail to be positive definite). Therefore the space
\lstinline{Lp E p μ} is defined as a subgroup of \lstinline{α →ₘ[μ] E}, the space of all
$\mu$-almost equal functions. Using almost everywhere equal functions is not very ergonomic since
pointwise operations are not well-behaved. For example, if \lstinline{f g : α →ₘ[μ] E} then
\lstinline{(f + g) x} is not equal to \lstinline{f x + g x}. Using the tactic
\lstinline{filter_upwards} it is possible to use pointwise operations to deduce almost everywhere
equality, but this cannot be combined with more powerful automation, such as \lstinline{ring} to
automatically prove equalities. For this reason, we use bare functions \lstinline{α → E} together
with the predicate \lstinline{MemLp p μ} and we introduce a new predicate \lstinline{MemLpLoc p μ}
for local $L^p$ functions. As a consequence, we have to adjust the definition of input-output
properties. Namely, if an input-output map is now simply a map \lstinline{(α → E) → (α → F)}, then
we have to impose the condition that it maps (local) $L^p$ functions to (local) $L^p$ functions on
the level of the predicates. A second issue arising from the use of bare functions is that we do not
have access to the norm on $L^p$ functions. The replacement is the \emph{extended} $L^p$ norm
\lstinline{eLpNorm u p μ : ℝ≥0∞} that takes values in the extended non-negative reals
$\overline{\RR_+} \coloneqq \RR_+ \cup \set{\infty}$.

Thus, the definition of finite gain stability for functions becomes
\begin{lstlisting}[mathescape]
structure IsFiniteGainStableWith
    (f : (α → E) → (α → F)) (γ β : ℝ≥0) (s : ι → Set α)
    (p : ℝ≥0∞) (μ : Measure α) where
  memLpLoc :
    ∀ ⦃u⦄, MemLpLoc u p μ → MemLpLoc (f u) p μ
  stableWith : ∀ t u (_hu : MemLpLoc u p μ),
    eLpNorm (f u) p (μ.restrict <| s t) ≤
    γ * eLpNorm u p (μ.restrict <| s t) + β
\end{lstlisting}
A \lstinline{structure} in Lean is a record-type data type that can be used for bundling together
propositions (as it is done here) or a mixture of data and propositions. In particular, if all
fields of a \lstinline{structure} are \lstinline{Prop}-valued, then the \lstinline{structure} is
also \lstinline{Prop}-valued. We follow the \mathlib convention that recommends prefixing
\lstinline{Prop}-valued structures and definitions with \lstinline{Is}.

The first field of the structure asserts that \lstinline{f} maps local $L^p$ functions to local
$L^p$ functions. Taking the \lstinline{eLpNorm} of \lstinline{u} with respect to the restricted
measure \lstinline{μ.restrict s} is the same as taking the \lstinline{eLpNorm} of
\lstinline{s.indicator u} with respect to the measure \lstinline{μ}, but the former is in some cases
easier to work with, so we use that in the definition.

With these definitions, the proof that the composition of finite gain stable maps is again finite
gain stable is remarkably simple:
\begin{lstlisting}[mathescape]
theorem comp
    (hg : g.IsFiniteGainStableWith γ' β' s p μ)
    (hf : f.IsFiniteGainStableWith γ β s p μ) :
    (g ∘ f).IsFiniteGainStableWith (γ * γ') (β * γ' + β') s p μ where
  memLpLoc u hu := hg.memLpLoc (hf.memLpLoc hu)
  stableWith t u hu := calc
    _ ≤ γ' * eLpNorm (f u) p _ + β' :=
      hg.stableWith t (f u) (hf.memLpLoc hu)
    _ ≤ γ' * (γ * eLpNorm u p _ + β) + β' := by
      gcongr; exact hf.stableWith t u hu
    _ = _ := by
      push_cast; ring
\end{lstlisting}

Now we turn to the formalization of feedback connections and the small-gain theorem.
First, the definition of finite gain stability of a relation is  straightforward:
\begin{lstlisting}[mathescape]
def IsFiniteGainStableWith
    (f : SetRel (α → E) (α → F)) (γ β : ℝ≥0)
    (s : ι → Set α) (p : ℝ≥0∞) (μ : Measure α) : Prop :=
  ∀ t u y (_hu : MemLpLoc u p μ) (_hy : MemLpLoc y p μ)
    (_h : (u, y) ∈ f), eLpNorm y p (μ.restrict <| s t) ≤
    γ * eLpNorm u p (μ.restrict <| s t) + β
\end{lstlisting}
We remark that we do \emph{not} assume for a finite gain stable relation $R$ that if $(u, y) \in R$
and $u \in L^p_{\loc}$, then $y \in L^p_{\loc}$.

Next, we define a \lstinline{structure} for the feedback connection, which is purely given in terms
of relations:
\begin{lstlisting}[mathescape]
structure SetRel.closedLoop where
  topRel : SetRel (α → E) (α → F)
  botRel : SetRel (α → F) (α → E)
\end{lstlisting}
where \lstinline{topRel} corresponds to $G_1$ and \lstinline{botRel} corresponds to $G_2$.
The relations $R_1$ and $R_2$ of \eqref{eq:R1} and \eqref{eq:R2} are then given by
\begin{lstlisting}[mathescape]
def inputState (loop : SetRel.closedLoop α E F) :
    SetRel (α → E × F) (α → E × F) :=
  {(e, u) | (fst ∘ u, snd ∘ u - snd ∘ e) ∈ loop.topRel ∧
    (snd ∘ u, fst ∘ e - fst ∘ u) ∈ loop.botRel }
\end{lstlisting}
\begin{lstlisting}[mathescape]
def inputOutput (loop : SetRel.closedLoop α E F) :
    SetRel (α → E × F) (α → F × E) :=
  {(e, y) | (fst ∘ e - snd ∘ y, fst ∘ y) ∈ loop.topRel ∧
    (snd ∘ e + fst ∘ y, snd ∘ y) ∈ loop.botRel }
\end{lstlisting}
if \lstinline{loop.topRel} is the graph of $G_1$ and \lstinline{loop.botRel} is the graph of $G_2$.

There is a subtle question about the norm on the product: if \lstinline{E} and \lstinline{F} are
normed spaces, then \lstinline{E × F} is a normed space equipped with the max-norm. However, in the
small-gain theorem we want to use the norm $L^p(\RR_+; E \times F)$. To do this, we use the
type-synonym \lstinline{WithLp} that equips the product with the $L^p$-norm and we define variants
\lstinline{inputStateLp} and \lstinline{inputOutputLp} that use \lstinline{WithLp p (E × F)}. We
also need the following two estimates that relate the \lstinline{eLpNorm} of
\lstinline{x ↦ (f x, g x)} to the \lstinline{eLpNorm} of \lstinline{f} and \lstinline{g}:

\begin{lstlisting}[mathescape]
theorem eLpNorm_withLp_prod_le_add (hp : 1 ≤ p)
    (hf : AEStronglyMeasurable f μ) :
    eLpNorm (fun x ↦ WithLp.toLp p (f x, g x)) p μ ≤ eLpNorm f p μ + eLpNorm g p μ
\end{lstlisting}
\begin{lstlisting}[mathescape]
theorem add_le_eLpNorm_withLp_prod (hp : 1 ≤ p)
    (hf : AEStronglyMeasurable f μ) :
    eLpNorm f p μ + eLpNorm g p μ ≤ addLEConst p *
      eLpNorm (fun x ↦ WithLp.toLp p (f x, g x)) p μ
\end{lstlisting}

With these lemmas we can prove the small-gain theorem without 
assuming any well-posedness of the feedback connection as stated below in Lean.
\begin{lstlisting}[mathescape]
theorem inputStateLp_isFiniteGainStableWith
    (hp : 1 ≤ p) (hG₁ : G₁.graph = loop.topRel)
    (hG₁' : G₁.IsFiniteGainStableWith γ₁ β₁ s p μ)
    (hG₂ : G₂.graph = loop.botRel)
    (hG₂' : G₂.IsFiniteGainStableWith γ₂ β₂ s p μ)
    (hγ : γ₁ * γ₂ < 1)
    (ht : ∀ t, MeasurableSet (s t) ∧ IsBounded (s t)) :
    (loop.inputStateLp p).IsFiniteGainStableWith
    (inputStateLoopGainLp p γ₁ γ₂)
    (loopBias γ₁ γ₂ β₁ β₂) s p μ
\end{lstlisting}
\begin{lstlisting}[mathescape]
theorem inputOutputLp_isFiniteGainStableWith
    (hp : 1 ≤ p) (hG₁ : G₁.graph = loop.topRel)
    (hG₁' : G₁.IsFiniteGainStableWith γ₁ β₁ s p μ)
    (hG₂ : G₂.graph = loop.botRel)
    (hG₂' : G₂.IsFiniteGainStableWith γ₂ β₂ s p μ)
    (hγ : γ₁ * γ₂ < 1)
    (ht : ∀ t, MeasurableSet (s t) ∧ IsBounded (s t)) :
    (loop.inputOutputLp p).IsFiniteGainStableWith
    (inputOutputLoopGainLp p γ₁ γ₂)
    (loopBias γ₁ γ₂ β₁ β₂) s p μ
\end{lstlisting}
The three constants \lstinline{inputStateLoopGainLp}, \lstinline{inputOutputLoopGainLp}, and
\lstinline{loopBias} are explicitly given and the first two provide bounds on the loop gain.

The first theorem states that the relation $R_1$ is finite gain stable and the second theorem
states that $R_2$ is finite gain stable, which are exactly the two claims in
Theorem~\ref{thm:small_gain}.

\section{Library architecture, Documentation, and Mathlib integration}

Currently, we have four top-level folders: \texttt{Basic}, \mathlib, \texttt{InputOutput}, and
\texttt{Stability}.
The \texttt{Basic} folder collects all abstract definitions and theorems, such as comparison
functions, abstract definitions of solution operators, and local $L^p$ functions.
The \texttt{Mathlib} folder contains results that are planned or in the process of being upstreamed
to \mathlib, such as missing lemmas in topology and ODEs.
The two remaining folders, \texttt{InputOutput} and \texttt{Stability}, contain the results on
input-output analysis and stability of dynamical systems, respectively.

In addition to the small-gain theorem, the folder on input-output analysis contains definitions of
causality, $L^p$-stability, and passive input-output systems.
We also give a characterization of finite gain stability for causal systems.

The folder on stability theory of dynamical systems contains various versions of Lyapunov's
theorem, proving stability for sets and points.
We also provide variants of the theorem for the case that the flow $\Phi$ is
given by the solution to the ODE $\dot{x} = f(x)$ and show that
the decay can be deduced from the negativity of the Lie derivative of $V$ with respect to $f$.
We define invariant sets and limit sets, and we prove LaSalle's invariance principle using the same
argument as in Cohen--Rouhling~\cite{CoRo}.
As in the case of Lyapunov's theorem, we have both a variant of LaSalle's principle for asymptotic
stability of sets and points.

For checking the usability of the definitions and theorems, we provide simple examples for both the
Lyapunov stability and input-output stability.
We prove that the origin for the dynamical system $\dot{x} = -r x$ is stable for $r \geq 0$ and
asymptotic stable for $r > 0$, as well as that a global minimum of a Hamiltonian dynamical system
is stable.
For the input-output analysis, we prove that the multiplication operator with an almost everywhere
bounded function is finite gain stable.

Documentation is a crucial part of any software library.
In this case, documentation can also serve as a reference for machine-checked theorems in control
theory.
We use \href{https://verso.lean-lang.org/}{Verso}\footnote{\url{https://verso.lean-lang.org/}} for
documentation, which can be found at \cite{LeanDynamicalSystems}.
Verso highlights some of the immediate advantages of digitalized mathematics:

All code examples in the documentation are checked by the Lean kernel, so if a proof is presented in
the documentation it is correct. It is possible to hover over definitions, theorems, and tactics to
see their documentation. Moreover, since the documentation is part of the library and checked during
continuous integration, we avoid the case that the documentation and the \texttt{main} branch
diverge, so for example if one commit renames a theorem that is mentioned in the documentation, CI
would fail to build the documentation unless the documentation is updated.

One advantage of the Lean ecosystem is that most mathematical objects are available from a single
integrated source, \mathlib. This avoids having to choose a particular implementation of a basic
object such as the real numbers, as opposed to~\cite{CoRo}. However, \mathlib does not yet have all
mathematics needed for formalizing dynamical system theory. We had to develop our own definition of
locally $L^p$ spaces and fill small gaps in the API for the \lstinline{nhdsSet} filter. We are
working on upstreaming these results to \mathlib.

Moreover, \mathlib is missing global existence results for ODEs and therefore it is currently not
possible to prove a sorry-free result, i.e., a proof with no admitted gaps, about stability of
fixed points of ODEs. On the other hand, \mathlib does not have the Carath\'eodory existence result
which is crucial for defining input-output systems from state systems. We plan to develop these
theories and move them to \mathlib eventually.

\section{Conclusion}

We laid the foundation for a monolithic library of formalized control theory that covers both
functional-analytic as well as calculus-based approaches. Using Lean and \mathlib has the advantage
of being able to employ a broad range of formalized pure mathematics that is important for this
task. Even though \mathlib currently does not contain all mathematics needed in control theory, its
use of abstractions makes it possible to develop the missing pieces and contribute results back to
\mathlib. We approach the formalization of applied mathematics in a similar abstract framework to
simplify the formalization and find the most general versions of the definitions and theorems used.
We hope that this abstract approach to applied mathematics also leads to new results in the area and
accelerates the development of trusted control systems. More concretely, our next immediate goal is
to expand the library by formalizing passivity theorems.

%% file: feedback2.tex
\begin{tikzpicture}

\tikzstyle{block} = [rectangle, minimum width=3cm, minimum height=1cm, text centered, draw=black]
\tikzstyle{sum} = [circle, minimum size=0.7cm, text centered, draw=black]
\tikzstyle{outDot} = [circle, minimum size=0.01cm, fill=black]

\tikzstyle{arrow} = [thick,->,>=stealth]

\node[sum] (topSum) {$-$};
\node[left=of topSum] (topIn) {$e_1$};
\node[block, right=of topSum] (G1) {$G_1$};
\node[block, below=of G1] (G2) {$G_2$};
\node[sum, right=of G2] (botSum) {$+$};
\node[right=of botSum] (botIn) {$e_2$};
\node[right=of G1] (topConn) {};
\node[right=of topConn] (topOut) {$y_1$};
\node[left=of G2] (botConn) {};
\node[left=of botConn] (botOut) {$y_2$};

%\node[outDot, right=of G1] (topOut) {};

%\draw[arrow] (2,-2) -- (botSum);
\draw[arrow] (topSum) -- (G1);
\draw[arrow] (G1) -| (botSum);
\draw[arrow] (botSum) -- (G2);
\draw[arrow] (G2) -| (topSum);
\draw (topOut) -- (topConn);
\draw (botOut) -- (botConn);
\draw[arrow] (botIn) -- (botSum);
\draw[arrow] (topIn) -- (topSum);

\end{tikzpicture}

%% file: FormalStability_arxiv.bbl
\begin{thebibliography}{10}

\bibitem{BiAn2025}
M.~Bin and D.~Angeli.
\newblock On an abstraction of lyapunov and lagrange stability.
\newblock {\em IFAC-PapersOnLine}, 59(19):412--417, 2025.
\newblock 13th IFAC Symposium on Nonlinear Control Systems NOLCOS 2025.

\bibitem{BiPa2023}
M.~Bin and T.~Parisini.
\newblock A small-gain theory for abstract systems on topological spaces.
\newblock {\em IEEE Transactions on Automatic Control}, 68(8):4494--4507, 2023.

\bibitem{Bourbaki}
N.~Bourbaki.
\newblock {\em General Topology. {Chapters} 1--4}.
\newblock Elements of Mathematics (Berlin). Springer-Verlag, Berlin, 1998.

\bibitem{CoRo}
C.~Cohen and D.~Rouhling.
\newblock A formal proof in {Coq} of {LaSalle}'s invariance principle.
\newblock In {\em Interactive Theorem Proving (ITP 2017)}, pages 148--163.
  Springer, 2017.

\bibitem{Lean}
L.~de~Moura and S.~Ullrich.
\newblock The {Lean} 4 theorem prover and programming language.
\newblock In {\em Automated Deduction -- {CADE}-28}, volume 12699 of {\em
  Lecture Notes in Comput. Sci.}, pages 625--635, Cham, 2021. Springer.

\bibitem{LeanDynamicalSystems}
M.~Doll and I.~Shames.
\newblock Nonlinear dynamical systems \& control theory.
\newblock \url{https://mcdoll.github.io/DynamicalSystems/}, 2026.

\bibitem{KeYmaeraX}
N.~Fulton, S.~Mitsch, J.-D. Quesel, M.~V{\"o}lp, and A.~Platzer.
\newblock {KeYmaera X}: An axiomatic tactical theorem prover for hybrid
  systems.
\newblock In {\em Automated Deduction -- {CADE}-25}, volume 9195 of {\em
  Lecture Notes in Comput. Sci.}, pages 527--538, Cham, 2015. Springer.

\bibitem{CSLib}
C.~Henson and F.~Montesi.
\newblock Computer science as infrastructure: the spine of the {Lean} computer
  science library ({CSLib}).
\newblock arXiv:2602.15078, 2026.

\bibitem{JaVe}
O.~A. Jasim and S.~M. Veres.
\newblock Towards formal proofs of feedback control theory.
\newblock In {\em 21st International Conference on System Theory, Control and
  Computing (ICSTCC)}, pages 43--48, 2017.

\bibitem{Khalil}
H.~K. Khalil.
\newblock {\em Nonlinear Systems}.
\newblock Macmillan, New York, 1992.

\bibitem{VeriDrone}
G.~Malecha, D.~Ricketts, M.~M. Alvarez, and S.~Lerner.
\newblock Towards foundational verification of cyber-physical systems.
\newblock In {\em 2016 Science of Security for Cyber-Physical Systems Workshop
  (SOSCyPS)}, pages 1--5, 2016.

\bibitem{Mathlib}
{The mathlib Community}.
\newblock The {Lean} mathematical library.
\newblock In {\em Proceedings of the 9th {ACM} {SIGPLAN} International
  Conference on Certified Programs and Proofs ({CPP} 2020)}, pages 367--381,
  New York, NY, USA, 2020. Association for Computing Machinery.

\bibitem{vanDerSchaft}
A.~J. van~der Schaft.
\newblock {\em {$L_2$}-Gain and Passivity Techniques in Nonlinear Control}.
\newblock Communications and Control Engineering. Springer, London, 3 edition,
  2016.

\bibitem{Wadler}
P.~Wadler.
\newblock Propositions as types.
\newblock {\em Communications of the ACM}, 58(12):75--84, 2015.

\bibitem{Zames}
G.~Zames.
\newblock On the input-output stability of time-varying nonlinear feedback
  systems. {Part} {I}: Conditions derived using concepts of loop gain,
  conicity, and positivity.
\newblock {\em {IEEE} Transactions on Automatic Control}, 11(2):228--238, 1966.

\end{thebibliography}
